\documentclass[]{amsart}
\usepackage{amsthm}
\usepackage{amssymb}
\usepackage{amsfonts}
\usepackage{tikz}

\newtheorem{thm}{Theorem}[section]

\newtheorem{lem}[thm]{Lemma}
\newtheorem{cor}[thm]{Corollary}

\theoremstyle{definition}

\newtheorem{definition}[thm]{Definition}
\newtheorem{example}[thm]{Example}

\newcommand{\seq}[1]{\langle #1\rangle}
\newcommand{\Z}{\mathbb Z}
\newcommand{\N}{\mathbb N}

\begin{document}
\title[Limits and Transitivity in Free Groups]{Limit Sets and Internal Transitivity in Free Group Actions}

\author[K. Binder]{Kyle Binder}
\address[K. Binder]{Baylor University, Waco TX, 76798}
\email[K. Binder]{Kyle\_Binder@baylor.edu}
\author[J. Meddaugh]{Jonathan Meddaugh}
\address[J. Meddaugh]{Baylor University, Waco TX, 76798}
\email[J. Meddaugh]{Jonathan\_Meddaugh@baylor.edu}

\subjclass[2010]{37B50, 37B10, 37B20, 54H20}
\keywords{shadowing, pseudo-orbit tracing property, topological dynamics}

\begin{abstract} It has been recently shown that, under appropriate hypotheses, the $\omega$-limit sets of a dynamical system are characterized by internal chain transitivity. In this paper, we examine generalizations of these ideas in the context of the action of a finitely generated free group or monoid. We give general definitions for several types of limit sets and analogous notions of internal transitivity. We then demonstrate that these limit sets are completely characterized by internal transitivity properties in shifts of finite type and general dynamical systems exhibiting a form of the shadowing property.
\end{abstract}

\maketitle

\section{Introduction}

The study of group and semi-group actions on compact metric spaces is a natural generalization of the study of traditional discrete dynamical systems. Unsurprisingly, as the groups of interest increase in complexity, basic results from the standard theory of topological dynamical systems (i.e $\Z$ and $\N$ actions) become difficult to generalize or fail entirely. What is particularly surprising, however, is just how quickly things begin to fall apart. Even in the case of the action of the group $\Z^d$ for $d\geq 2$, simple questions regarding entropy are difficult to answer even in the simplified context of shift spaces \cite{hochman1, hochman2}.

However, there are aspects of these actions that have been studied with some success. Of particular interest to this paper are generalizations of the known relationship between $\omega$-limit sets and internal chain transitivity, especially under the additional hypothesis of shadowing. In particular, it has long been known that $\omega$-limit sets are internally chain transitive \cite{Hirsch}, and that in certain categories of dynamical systems, the converse is also true including shifts of finite type, topologically hyperbolic maps, certain Julia sets of quadratic polynomials and Axiom A diffeomorphisms  \cite{BGKR-omega, BGOR-DCDS, BR-ToAppear, BMR-ToAppear, Bowen}.  More recently, Meddaugh and Raines demonstrated that under the assumption that the system has the \emph{shadowing property}, the collection of internally chain transitive sets is precisely the closure of the collection of $\omega$-limit sets in the Hausdorff topology \cite{MR}. Thus, in situations in which it is known, a priori, that the collection of $\omega$-limit sets is closed in this topology (interval maps are an important example of such a situation \cite{Blokh}), the shadowing property is sufficient to establish that internal chain transitivity coincides precisely with the property of being an $\omega$-limit set. This connection is further explored in \cite{GoodMeddaugh-ICT}, in which variations of the shadowing property are explored to more precisely characterize the necessary conditions under which this equivalence occurs.

There have been a number of attempts to extend these sorts of results to the broader context of group actions. In particular, several authors have considered the structure of limit sets for $\Z^d$ actions, specifically for subshifts of $\Sigma^{\Z^d}$ where $\Sigma$ is a finite alphabet. Among the first generalizations along theses lines is the work of Oprocha \cite{oprochalimitsets, oprochashadowingmulti} in which notions of shadowing and limit sets for these sorts of actions are explored. The connection between limit sets and internal transitivity was explored by Meddaugh and Raines \cite{MR-Zd}. An important aspect of this body of work is that for multi-dimensional shift spaces, there is no unique natural analog for forward iteration, so many types of limit sets arise, and correspondingly, many notions of internal transitivity. One other common aspect of these results is that the problem of \emph{completability} arises quite frequently. This is not surprising, as it is related to tiling problems, in which this issue is well-known problem  \cite{Berger}.

It is at this point that we propose that $\Z^d$ actions are not necessarily the most natural successors to $\Z$ actions in terms of complexity. In this paper, we examine the action of finitely generated free groups and monoids on compact metric spaces. We examine notions of limit sets and establish that there are internal transitivity properties that characterize them in the context of shifts of finite type. We then introduce a generalization of the notion of shadowing to these actions and establish that a subshift has shadowing precisely when it is a shift of finite type. This parallels the known result for $\Z$ actions as demonstrated by Walters \cite{WaltersFiniteType}. We then go on to demonstrate that under the assumption of shadowing, we witness analogs to the results of Meddaugh and Raines \cite{MR}, and, under an appropriate analog of \emph{asymptotic shadowing}, we recover generalizations of the characterization of limit sets in terms of internal transitivity as developed by Good and Meddaugh \cite{GoodMeddaugh-ICT}.

\section{Preliminaries}
Let $G$ be a group or monoid and $X$ a compact metric space. A continuous left $G$-action on $X$ is a function $f:G\times X\to X$ that satisfies the following conditions: for each $g\in G$, the function $f_g$ defined by $f_g(x)=f(g,x)$ is continuous, for the identity $e\in G$, $f_e$ is the identity on $X$, and for $g,h\in G$, $f_{gh}=f_h\circ f_g$.

For $n\in\mathbb{N} $, let $ F $ denote the free group on the $n$ generators $\{s_0,\ldots s_{n-1}\}$ and $ S=\{s_0,\ldots s_{n-1},s_0^{-1},\ldots s_{n-1}^{-1}\} $ be the set of all generators and their inverses. The set of reduced words of $F$ is the set $W=\{e\}\cup\{w_0w_1\cdots w_k\in S^\omega: w_i\neq w_{i+1}^{-1} \textrm{ for } i<k\}$ with $e$
denoting the empty word. Each element of $F$ has a unique representative in $W$ and the group operation of $F$ is realized in $W$ by concatenation followed by cancellation.

We define the \emph{length} of an element $ u \in F $ to be the number of letters in its reduced representation. We denote this by $ \left| u \right| $.
The identity of $ F $ is associated with the empty word $ e $ and has length $ 0 $.
Finally, we will say for two elements of $ F $ with reduced representations $ u = u_{0} \dots u_{n} $ and $ v = v_{0} \dots v_{n+k} $ , that $ u $ is a  \emph{prefix} of
$ v $ if $ u_{i} = v_{i} $ for $ 0 \leq i \leq n $. Additionally, we will take the identity $ e $ to be a prefix of every element of $ F $.

It will also be necessary to consider the infinite words of $F$. In particular, we define the set $W_\infty=\{\seq{w_i}_{i\in\omega}: w_i\neq w_{i+1}^{-1} \textrm{ for } i\in\omega\}$. Let $ u, w $ be two words (either finite or infinite) with length at least $ n $. We say $ u|_{n} = w|_{n} $ if $ u_{i} = w_{i} $ for $ 0 \leq i < n $.

For $ n \in \mathbb{N} $, let $ H $ denote the free monoid on the $ n $ generators $ P = \{s_0,\ldots s_{n-1}\}$.
In this case, $H$ coincides with the set of words of $\{e\}\cup\{w_0w_1\cdots w_k \in P^{\omega} \}$ with $ e $ denoting the empty word as
every element of $ H $ has a unique representation, and the binary operation is simply concatenation.
The collection of infinite words of $ H $ is the set $ H_{\infty} = \{\seq{w_i}_{i\in\omega}: w_{i} \in P\} $.
The length of elements, prefixes, and restrictions $ w_{|n} $ are defined the same as in the free group case.

For the sake of generality, we take $ G $ be either a free group or monoid with $ W $ the set of words,
$ W_{\infty} $ the set of infinite words, and $ S $ the set of generators (with inverses in the
group case).

For a finite alphabet $ \mathcal{A} $, \emph{the full shift of $ \mathcal{A} $ over $ G $} (denoted $ \mathcal{A}^{G} $) is the set
of all functions $ x: G \rightarrow \mathcal{A} $. There is a natural $G$-action $\sigma$ on $\mathcal A^{G}$ defined as follows.
For every $ s \in S $ there is an associated \emph{shift map} $ \sigma_{s}: \mathcal{A}^{G} \rightarrow \mathcal{A}^{G} $ defined by
$ \sigma_{s}(x)(u) = x(su) $.
For $ v = v_{0} \dots v_{n} \in G $, we define $ \sigma_{v}(x) = \sigma_{v_{n}} \circ \dots \circ \sigma_{v_{0}}(x) $. Thus $ \sigma_{v}(x)(u) = x(vu) $.

For fixed $ G $, define $ \Sigma^{n} = \left\{u \in G : \left|u \right| < n \right\} $.
We can place a metric on $ \mathcal{A}^{G} $ defined by $ d(x,y) = \inf\left( \left\{ 2^{-n} : x|_{\Sigma^{n}} = y|_{\Sigma^{n}}  \right\}\cup\{1\}\right) $.
It is easy to see that under the topology induced by this metric, $\mathcal A^{G}$ is compact and the action $\sigma$ is continuous. In particular, if $d(x,y) = 2^{-n} $ then $ d(\sigma_{i}(x), \sigma_{i}(y)) \leq 2^{-n + 1} $.
Furthermore, for $ u \in G $, if $ d(x,y) = 2^{-n} $ then $ d(\sigma_{u}(x), \sigma_{u}(y)) \leq \min \left\{ 2^{-m + \left| u \right|}, 1 \right\} $.

An \emph{n-block} is a function $ B_{n}: \Sigma^{n} \rightarrow \mathcal{A} $.
An element $ x \in \mathcal{A}^{G} $ is said to contain $ B_{n} $ if there exists a $ u \in G $ such that $ \sigma_{u}(x)|_{\Sigma^{n}} = B_{n} $.
Let $ \mathcal{F} $ be a collection of m-blocks where m is allowed to range over the integers.
We can create a subspace of $ \mathcal{A}^{G} $ defined as $ X_{\mathcal{F}} = \left\{x \in \mathcal{A}^{G} : x \; \textrm{does not contain any} \; B \in \mathcal{F} \right\} $.
$ X_{\mathcal{F}} $ is compact and invariant under the shift maps.
In this case, $ \mathcal{F} $ is called a \emph{set of forbidden blocks}.
Any invariant and compact subspace of $ \mathcal{A}^{G} $ can be expressed as $ X_{\mathcal{F}} $ for some set $ \mathcal{F} $.
This is called a \emph{shift space}.
In the case that $ \mathcal{F} $ is finite, the shift space is a \emph{shift of finite type} (SFT).
If $ M $ is the maximal integer such that some $ B \in \mathcal{F} $ is an M-block, $ Y $ is called an \emph{M-step shift of finite type}.
In this case, we can assume without loss of generality that every element of $ \mathcal{F} $ is an M-block.

\begin{figure}
	\centering

	\begin{tikzpicture}[scale = .3]
	\draw (0,0) -- (0,(4) -- (0,6);
	\draw[dotted] (0,6) -- (0,7.25);
	\draw (0,0) -- (4,0) -- (6,0);
	\draw[dotted] (6,0) -- (7.25,0);
	\draw (0,0) -- (0,-4) -- (0, -6);
	\draw[dotted] (0,-6) -- (0, -7.25);
	\draw (0,0) -- (-4,0) -- (-6,0);
	\draw[dotted] (-6,0) -- (-7.25, 0);
	\draw (-2,4) -- (0,4) -- (2, 4);
	\draw (-2,-4) -- (0,-4) -- (2, -4);
	\draw (4,2) -- (4,0) -- (4, -2);
	\draw (-4,2) -- (-4,0) -- (-4, -2);
	\draw[dotted] (1.25,6) -- (0,6) -- (-1.25,6);
	\draw[dotted] (1.25,-6) -- (0,-6) -- (-1.25,-6);
	\draw[dotted] (6,1.25) -- (6,0) -- (6,-1.25);
	\draw[dotted] (-6,1.25) -- (-6,0) -- (-6,-1.25);
	\draw[dotted] (2.25,4) -- (2,4);
	\draw[dotted] (2.25,-4) -- (2,-4);
	\draw[dotted] (2, 2.25) -- (2,4) -- (2,5.25);
	\draw[dotted] (2, -2.25) -- (2,-4) -- (2,-5.25);
	\draw[dotted] (-2, 2.25) -- (-2,4) -- (-2,5.25);
	\draw[dotted] (-2, -2.25) -- (-2,-4) -- (-2,-5.25);
	\draw[dotted] (4, 2) -- (4,2.25);
	\draw[dotted] (-4, 2) -- (-4,3.25);
	\draw[dotted] (4, -2) -- (4,-3.25);
	\draw[dotted] (-4, -2) -- (-4,-3.25);
	\draw[dotted] (5.25,2) -- (4,2) -- (2.25,2);
	\draw[dotted] (-5.25,2) -- (-4,2) -- (-2.25,2);
	\draw[dotted] (5.25,-2) -- (4,-2) -- (2.25,-2);
	\draw[dotted] (-5.25,-2) -- (-4,-2) -- (-2.25,-2);
	\draw[dotted] (-3.25,-4) -- (-2,-4);
	\draw[dotted] (-3.25,4) -- (-2,4);

	\node at (0,0) [fill=white] {\tiny 0};
	\node at (4,0) [fill = white]{\tiny 1};
	\node at (-4,0) [fill = white]{\tiny 0};
	\node at (0,4) [fill = white]{\tiny 0};
	\node at (0,-4) [fill = white]{\tiny 1};
	\node at (2,4) [fill = white]{\tiny 0};
	\node at (2,-4) [fill = white]{\tiny 0};
	\node at (-2,4) [fill = white]{\tiny 1};
	\node at (-2,-4) [fill = white]{\tiny 1};
	\node at (4,2) [fill = white]{\tiny 1};
	\node at (4,-2) [fill = white]{\tiny 1};
	\node at (-4,2) [fill = white]{\tiny 1};
	\node at (-4,-2) [fill = white]{\tiny 1};
	\node at (0,6) [fill = white]{\tiny 1};
	\node at (0,-6) [fill = white]{\tiny 1};
	\node at (6,0) [fill = white]{\tiny 1};
	\node at (-6,0) [fill = white]{\tiny 1};
	\end{tikzpicture}
	\caption{A representation of an element of $ \left\{0,1 \right\}^{F_{2}} $ with the middle 3-block filled in.} \label{fig:M1}

\end{figure}

In a metric space $ X $, the distance between a point $ x \in X $ and closed subset $ A \subseteq X $ we define as
$ d(x,A) = \inf_{a \in A}d(x,a) $. This gives rise to the \emph{Hausdorff distance} between two closed subsets $ A, B \subseteq X $
defined as
\[ d_{H}(A,B) = \max\{\sup_{a \in A}d(a,B), \sup_{b \in B}d(b,A)\} .\]

\section{$\omega$-limit Sets for Group Actions}\label{Omega}

In the context of continuous group or monoid actions on a compact metric space $X$, there is no clear \emph{best} analogue to the notion of an $\omega$-limit set. Principal to this ambiguity is that there is more than one `future' to consider. This problem is especially apparent in the action of groups, but is still present even if we restrict our attention to monoids. In \cite{MR-Zd} and \cite{oprochalimitsets}, the authors explore some of these notions in the context of $\mathbb Z^d$ actions. Perhaps the most general notion of limit sets in this context are the semigroup limit sets as defined by Barros and Souza \cite{Barros2010}.

\begin{definition}
	Let $S$ be a semigroup acting continuously on $X$ and let $\mathcal F$ be a family of subsets of $S$. For $x\in X$, we define
	\[\omega(x,\mathcal F)=\bigcap_{A\in\mathcal F}\overline{\{f_s(x) : s\in A\}}\]
\end{definition}

It is immediate that $\omega(x,\mathcal F)$ is a compact subset of $X$. If $\mathcal F$ is taken to be a filter base (for any $A,B\in\mathcal F$, there exists $C\in\mathcal F$ with $C\subseteq A\cap B$), then $\omega(x,\mathcal F)$ is nonempty, and under appropriate assumptions, it is also invariant \cite{Barros2012}.

For the purposes of the following sections, we consider only the action of free groups and free monoids with finitely many generators. In particular, let $G$ be the free group or monoid on $n$ generators, with $S$ the set of generators (and inverses in the case $ G $ is a free group) and let $G$ act continuously on the compact metric space $X$ by the assignment $u\in G\mapsto f_u:X\to X$. We also take $W$ to be the set of reduced words in $G$, and $W_\infty$ the set of infinite reduced words.

In this setting, there are a number of specific limit set types that warrant particular attention.

The first of these is acquired by taking $\mathcal{F}$ to be the family of subsets of words with minimum lengths.

\begin{definition}
	For $ x \in X $, $ \omega(x) = \bigcap_{n \in \omega} \overline {\left\{ f_{u}(x) : \left| u \right| > n \right\}} $.
\end{definition}

The equivalent metric formulation is given in the following lemma.

\begin{lem}
	$ \omega(x) = \left\{y \in X : \forall \, n \in \mathbb{N} \; \; \exists \, \left| u_{n} \right| > n  \; \; s.t. \; \;  d(f_{u_{n}}(x), y) < \frac{1}{n}  \right\} $.
\end{lem}

\begin{proof}
	Let $ y \in \omega(x) $. For $ n \in \mathbb{N} $ there exists a sequence $ \left\{ u_{i} \right\}_{i \in \mathbb{N}}  \subseteq G $ with $ u_{i} > n $
	such that $ \left\{ f_{u_{i}}(x) \right\} $ converges to $ y $, as $ y \in \overline{ \left\{ f_{u}(x): \left| u \right| > n \right\} }$.
	Thus we can choose $ \left| u_{n} \right| > n $ with $ d (f_{u_{n}}(x), y) < \frac{1}{n} $.

	Now let $ y \in \left\{ y \in X : \forall \, n \in \mathbb{N} \; \; \exists \, \left| u_{n} \right| > n  \; \; s.t. \; \;  d(f_{u_{n}}(x), y) < \frac{1}{n}  \right\} $.
	For $ n \in \mathbb{N} $ choose $ \left| u_{n} \right| > n $ with $ d(f_{u_{n}}(x), y) < \frac{1}{n}  $.
	Thus, for $ m \geq n $, $f_{u_{m}}(x) \in \left\{f_{u}(x):\left| u \right| > n \right\} $, and $ \left\{ f_{u_{i}}(x) \right\}_{i = m}^{\infty} $ converges to $ y $.
	Therefore $ y \in \overline{\left\{f_{u}(x): \left| u \right| > n \right\}} $.
\end{proof}

\begin{lem}
	$ \omega(x) $ is invariant.
\end{lem}

\begin{proof}
	Let $ y \in \omega(x) $ and $ i \in S $. By the uniform continuity of $ f_{i} $, there exists $ \delta_{n} > 0 $ such that $ d(f_{i}(x), f_{i}(y)) < \frac{1}{n} $ if $ d(x, y) < \delta_{n} $.
	As $ y \in \omega(x) $, there is a sequence $ \left\{ u_{j} \right\} $  increasing in length such that $ d(y, f_{u_{j}}(x)) < \delta_{j} $.
	Hence, $ d(f_{i}(y), f_{u_{j}i}(x)) < \frac{1}{j} $ for all $ j $ so $ y \in \omega(x) $.
\end{proof}

This notion of limit set captures all possible `futures' of $x$ under the action of $G$, completely discarding any notion of direction. In contrast, the following type of limit sense has a very strong sense of directionality, capturing only those behaviors of the orbit of $x$ realized along a particular trajectory.

\begin{definition}
	For $ x \in X $ and $ w \in W_{\infty} $, $ \omega_{w}(x) = \bigcap_{n \in N} \overline {\left\{ f_{w_{0} \dots w_{k}}(x) : k > n \right\}}$.
\end{definition}

Here, we take the family $\mathcal{F}$ to be the collection of finite prefixes of the infinite word $w$. As expected, $\omega_w(x)$ is compact and non-empty. And, while it is not necessarily invariant, it does still have a limited level of invariance.

\begin{thm}
	In the case of a free group action, for $ y \in \omega_{w}(x) $ there exists $ i \neq j \in S $ such that $ f_{i}(y), \; f_{j}(y) \in \omega_{w}(x) $.
\end{thm}

\begin{proof}
	Let $ y \in \omega_{w}(x) $ be given.
	By the uniform continuity of the maps $ f_{i}$, for every $ n \in \mathbb{N} $ there is $ k_{n} $ such that
	if $ d(x,z) < \frac{1}{k_{n}} $, then $ d(f_{i}(x), f_{i}(z)) < \frac{1}{n} $ for all $ i \in S $.
	For $ n \in \mathbb{N} $ choose $ m_{n} > n + 1 $ such that $ d(f_{w_{0} \dots w_{m_{n}}}(x), y) < \frac{1}{k_{m}} $.
	Choose $ i \in S $ such that $ i = w_{m_{n} + 1} $ for infinitely many $ n $.
	We can then choose $ j $ so $ j = w_{m_{n}}^{-1} $ for infinitely many $ n $ with
	$ w_{m_{n} + 1} = i $.
	Thus $ i \neq j $.
	By passing to a subsequence if necessary, we can assume $ w_{m_{n}} = i^{-1}$ and $ w_{m_{n}+1}= j$.
	Notice then that $ d(f_{w_{0} \dots w_{l_{n}}}(x), y) < \frac{1}{k_{n}} $.
	Thus $ d(f_{j}(f_{w_{0} \dots w_{m_{n}}}(x)), f_{j}(y)) = d(f_{w_{0} \dots w_{m_{n} - 1}}(x), f_{j}(y)) < \frac{1}{n} $.
	Similarly, $ d(f_{w_{0} \dots w_{m_{n} + 1}}(x), f_{i}(y)) < \frac{1}{n} $.
	As $ m_{n} - 1 > n $, $ f_{i}(y) $, $ f_{j}(y) \in \omega_{w}(x) $.
\end{proof}

Because there are not inverses in a free monoid, we obtain a slightly weaker result in this
context.

\begin{cor}
	In the case of a free monoid action, for $ y \in \omega_{w}(x) $ there exists $ i \in S $ such that $ f_{i}(y) \in \omega_{w}(x) $ and there exists $z\in\omega_w(x)$ and $j\in S$ with $y=f_j(z)$.
\end{cor}

A similar limit set with a less rigid, but still quite strong, notion of directionality is the following. Here, the family $\mathcal F$ is the collection of subsets of words that share prefixes of specified lengths with $w$.

\begin{definition}
	For $ x \in X $ and $ w \in W_{\infty} $, $ \omega_{F_{w}}(x) =  \bigcap_{n \in N} \overline {\left\{ f_{u}(x) : u|_{n} = w|_{n} \right\} } $
\end{definition}

Again this is compact and non-empty and has an equivalent analytic definition.

\begin{lem}
	$ \omega_{F_{w}}(x) = \left\{y \in X : \forall \; m \; \exists \, u \in G \; s.t. \; d(f_{u}(x), y)) < \frac{1}{m} \textrm{ and } u_{|m} = w_{|m} \right\}  $.
\end{lem}

A principal benefit of this looser notion of directionality is that we recover invariance.

\begin{thm}
	$ \omega_{F_{w}}(x) $ is invariant.
\end{thm}

\begin{proof}
	For $ m \in \mathbb{N} $ choose $ k_{m} > m $ such that for $ i \in S $, if $ d(y,z) < \frac{1}{k_{m}} $ then $ d(f_{i}(y), f_{i}(z)) < \frac{1}{m} $.
	Now find $ u \in G $ such that $ d(f_{u}(x), y) < \frac{1}{k_{m}} $. Thus $ d( f_{ui}(x), f_{i}(y)) < \frac{1}{m} $. As $ k_{m} > m $, $ ui_{|m} = w_{|m} $.
	Therefore $ f_{i}(y) \in \omega_{F_{w}}(x) $.
\end{proof}

Among these limit set types, we have the following relationships.

\begin{thm}
	For all $w\in W_\infty$ and all $x\in X$, we have \[\omega_{w}(x) \subseteq \omega_{F_{w}}(x) \subseteq \omega(x).\]
\end{thm}

\begin{proof}
	The first inclusion follows from the observation that $ \left\{ f_{w_{0} \dots w_{k}}(x) : k > n \right\} \subseteq \left\{ f_{u}(x) : u_{|n} = w_{|n} \right\} $. The second inclusion similarly follows from the observation that $ \left\{ f_{u}(x) : u_{|n} = w_{|n} \right\}\subseteq \left\{ f_{u}(x) : \left| u \right| > n \right\}$.
\end{proof}
%
%
%
%
%
%
%

\section{Analogues of internal chain transitivity}

As explored in a number of papers \cite{BGOR-DCDS, BGKR-omega, BMR-Variations, BDG-tentmaps, GoodMeddaugh-ICT}, there is a strong connection between limit sets and notions of transitivity. It is therefore not surprising that there are connections between analogous notions in the dynamics of group actions on compact metric spaces. In particular, we begin by seeking analogues of chain transitivity that will characterize the classes of limit sets described in Section \ref{Omega}.
For notation, let $ \mathfrak{W}_w $ denote the collection of all $ \omega_{w}$-limit sets and $ \mathfrak{W}_{F_{w}} $ the collection of all $ \omega_{F_{w}}$-limit sets in $ X $.

\begin{definition}
	Given $ \epsilon > 0 $ and an element $ u = u_{1} \dots u_{n} \in W $, an $ \epsilon $\emph{-chain indexed by u} is
	a sequence $ \left\{x_{1}, \dots, x_{n+1} \right\} $ of $ X $ such that $ d(f_{u_{i}}(x_{i})), x_{i+1}) < \epsilon $.
\end{definition}

\begin{definition}
	A closed subset $Y$ of $X$ is \emph{internally chain transitive} ($Y\in ICT$) if
	for every $ x, y \in Y $ and $ \epsilon > 0 $ there is a $ u \in W $ and $ \epsilon $-chain indexed by $ u $ with
	$ x_{1} = x $ and $ x_{n+1} = y $.
\end{definition}

\begin{thm}
 $ ICT $ is closed.
\end{thm}

\begin{proof}
 Let $ Y \in \overline{ICT} $, $x,y \in Y $ and $ \epsilon > 0 $.
 By the uniform continuity of $ f_{i} $ there is a $ \frac{\epsilon}{3} > \delta > 0 $
 such that $ d(a,b) < \delta $ implies $ d(f_{i}(a), f_{i}(b)) < \frac{\epsilon}{3} $.
 Choose $ B \in ICT $ with $ d_{H}(A,B) < \delta $ and $ x', y' \in B $ with $ d(x,x') < \delta
 $ and $ d(y,y') < \delta $.
 There is a $ \delta $-chain indexed by $ u $ from $ x' $ to $ y' $ in $ B $.
 Say this chain is $ \left\{ x_{1}, \dots , x_{n+1} \right\} $.
 For $ 2 \leq i \leq n $ choose $ z_{i} \in Y $ with $ d(z_{i}, x_{i}) <  \delta $ and let $ z_{1} = x $, $ z_{n+1} = y $.
 It follows that $ \left\{z_{1}, \dots, z_{n+1} \right\} $ is an $ \epsilon $-chain from
 $ x $ to $ y $ as $ d(f_{u_i}(z_{i}), z_{i+1}) \leq d(f_{u_i}(z_{i}), f_{u_i}(x_{i})) + d(f_{u_i}(x_{i}), x_{i+1}) + d(x_{i+1}, z_{i+1})
 < \epsilon $. Therefore $ Y \in ICT $.
\end{proof}

We have the following results correlating limit sets with internal chain transitivity.

\begin{lem}\label{omegaSufClose}
	For every $ \epsilon > 0 $, $ w = w_{1}w_{2}\dots \in W_\infty $, and $ x \in X $ there exists $ N_{\epsilon} \in \mathbb{N} $ such that for $ n > N_{\epsilon} $ $ d(f_{w_{1}\dots w_{n}}(x), \omega_{w}(x)) < \epsilon $.
\end{lem}

\begin{proof}
	Suppose to the contrary. Then we have an increasing sequence of integers $ \left\{m_{n}\right\} $ with $ d(f_{w_{1} \dots w_{m_{n}}}(x), \omega_{w}(x)) > \epsilon $.
	By passing to a subsequence if necessary, $ \left\{f_{w_{1} \dots w_{m_{n}}} (x) \right\} $ converges to a point $ y \in \omega_{w}(x) $.
	However, $ d(y, \omega_{w}(x)) \geq \epsilon $, a contradiction.
\end{proof}

\begin{thm}
	$ \mathfrak{W}_{w} \subseteq ICT $.
\end{thm}

\begin{proof}
	Let $ \omega_{w}(x) \in \mathfrak{W}_{w} $, $ y,z \in \omega_{w}(x) $, and $ \epsilon > 0 $. By the uniform continuity of $ f_{i} $,  there is a $\frac{\epsilon}{3} > \delta > 0 $
	such that if $ d(p, q) < \delta $ then $ d(f_{u}(p), f_{u}(q)) < \frac{\epsilon}{3} $ for $ u \in S $.
	Let $ N_{\delta} $ be given by the previous lemma. Find $ n > m > N_{\delta} $ such that $ d(f_{w_{1} \dots w_{m}}(x), y) < \delta $ and $ d(f_{w_{1} \dots w_{n}}(x), z) < \delta $.

	Let $k=n-m$ and fix $ t_{1}\dots t_{k} = w_{m+1} \dots w_{n} $. Set $ x_{0} = y $, $ x_{k} = z $. For $1<i< k $ choose $ x_{i} $ so $ d(x_{i}, \sigma_{w_{1} \dots w_{n + i -1}}(x)) < \delta $.
	We claim that $ d(f_{t_{i}}(x_{i}), x_{i+1}) < \epsilon $. $ d(x_{i}, f_{w_{1} \dots w_{n + i -1}}(x)) < \delta $ so $d(f_{t_{i}}(x_{i}), f_{w_{1} \dots w_{n + i}}(x)) <
	\frac{\epsilon}{3}$. Thus $ d(f_{t_{i}}(x_{i}), x_{i+1}) < d(f_{t_{i}}(x_{i}),  f_{w_{1} \dots w_{n + i}}(x)) +  d(x_{i+1}, f_{w_{1} \dots w_{n + i}}(x)) < \epsilon$.
	Therefore $ \omega_{w}(x) $ is internally chain transitive.
\end{proof}

\begin{lem}\label{omegaFSufClose}
	Given $ x \in X $ and $ w \in W $, for every $ \epsilon > 0 $ there exists $ N_{\epsilon} \in \mathbb{N} $ such that for all $ u \in G $ with $ u_{|n} = w_{|n} $ for $ n > N_{\epsilon} $, $ d(f_{u}(x), \omega_{F_{w}}(x)) < \epsilon $.
\end{lem}

\begin{proof}
	Suppose not. Then for every $ n > N $ there is a $ u_{n} \in G $ with $ u_{n |n} = w_{|n} $ such that $ d(f_{u_{n}}(x), \omega_{F_{w}}(x)) > \epsilon $.
	By passing to a subsequence if necessary, these $ f_{u_{n}}(x) $ converge to a point $ y \in \omega_{F_{w}}(x) $. However, $ d(y, \omega_{F_{w}}(x)) > \epsilon $,
	a contradiction.
\end{proof}

With this lemma, we get the following theorem using the same technique as in the analogous theorem for $  \omega_{w}(x) $.

\begin{thm}
	For an $F$-action on $X$, $ \mathfrak{W}_{F_{w}} \subseteq ICT$.
\end{thm}

It should be noted that when $G$ is a group, internal chain transitivity is a rather weak condition. In particular, since there is no inherent `direction' for chains, any finite segment of an orbit is an internally chain transitive set. In particular, we fail to see the expected equivalence of being a limit set and being internally chain transitive. Even in the context of full shifts, internal chain transitivity alone is not enough to characterize $ \omega_{w}(x) $.

\begin{example}
	Let $ X $ be the full-shift on the alphabet $ \left\{0,1,2\right\} $ over $ F_{2} $.
	Let $ Y \subset X = \left\{x_{0}, x_{1}, \dots \right\}$ where $ x_{0} $ is 1 for elements of the form $ a^{m}b^{n} $ ($ n > 0 $) and 0 elsewhere,
	$ x_{1} = \sigma_{b^{-1}}(x_{0}) $. $ x_{2} $ is the same as $ x_{0} $ except $ 1 $ is 2. For $ i > 2 $, $ x_{2+i} = \sigma_{a^{i}}(x_{2})$.
	It is not hard to see that this is ICT. Getting to and from $ x_{0} $ and $ x_{1} $ is just a shift as is $ x_{2} $ to $ x_{2+i} $. To get from $ x_{1} $ to $ x_{2} $ we
	just need to get $ i $ sufficiently large so that $ d(\sigma_{a}(x_{1}), x_{2+i}) < \epsilon $ and then shift to get to $ x_{2} $.

	However, $ Y $ cannot be expressed as $ \omega_{w}(x) $ for any $ w, x $. Note that for $ \epsilon < 2^{-2} $ any $ \epsilon $-chain ending at $ x_{0} $ must
	be indexed by a word ending in $ b $ and any $ \epsilon $-chain beginning at $ x_{0} $ and going to $ x_{i} $ ($ i > 0 $) must begin with $ b^{-1} $. However, we claim this can
	never be the case for an element in $ \omega_{w}(x) $.

	Let $ x \in X $, $ 2^{-2} > \epsilon > 0 $, $ w = w_{1}w_{2} \dots \in W_{\infty} $ be given, and choose $ N_{\epsilon} \in \mathbb{N} $ such that for $ m > N_{\epsilon} $, $ d(\sigma_{w_{1} \dots w_{m}}(x), \omega_{w}(x)) < \epsilon $. Let $ \frac{\epsilon}{3} > \delta > 0$ so that if $ d(x, y)< \delta $ $ d(\sigma_{i}(x), \sigma_{i}(y)) < \frac{\epsilon}{3} $.
	Find $ k > N_{\epsilon} $ so that $ d(\sigma_{w_{1} \dots w_{k}}(x), x_{1}) < \delta  $, $ m > k $ with $ d(\sigma_{w_{1} \dots w_{m}}(x), x_{0}) < \delta $ and $ l > m $ with
	$ d(\sigma_{w_{1} \dots w_{l}}(x), x_{1}) < \delta $. As in the proof that $ \omega_{w}(x) $ is ICT, we can get an $ \epsilon $-chain between $ x_{1} $ and $ x_{0} $ indexed
	by the word $ w_{m+1}\dots s_{k} $ and a $ \epsilon $-chain  between $ x_{0} $ and $ x_{1} $ indexed by $ w_{k+1} \dots  w_{l} $. As $ w_{k} $, $ w_{k+1} $ are in a reduced word,
	$ w_{k} \neq w_{k+1}^{-1} $.
	This then shows that our example cannot be expressed as an $ \omega_w $-limit set.
\end{example}

With this example, we see that a stronger form of chain transitivity is necessary to characterize $\omega_w$-limit sets for free group actions.

\begin{definition}
	Let $ Y$ be a closed subset of $X$. $ Y $ is \emph{consistently internally chain transitive} ($Y\in CICT$) if for every $ x \in Y $ there are
	$ i(x), t(x) \in S $ with $ i(x) \neq t(x)^{-1} $ such that for every $ \epsilon > 0 $ and $ x, y \in Y $ there is an $ \epsilon $-chain
	between $ x $ and $ y $ indexed by a word starting with $ i(x) $ and ending with $ t(y) $.
\end{definition}

It is immediate from this definition that $CICT \subseteq ICT$. Also, in the context of free monoids, $ CICT $ is equal to $ ICT $
as there are no inverses.

In the case of $ ICT $, it was relatively easy to show the set is closed.
However, the analogous result for $ CICT $ has an added subtlety.
Not only must there be $ \delta $-chains, but also each point $ x $ has an associated initial and terminal index $ i(x) $ and $ t(x) $.
For $ Y \in \overline{CICT} $ and $ x \in Y $, it is intuitive to define $ i(x), j(x) $ to match those of points in $ CICT $ sets that
converge to $ x $. However, defining $ i(x), j(x) $ for all $ x $ in this manner may not imply $ i(x), t(y) $ interact in the necessary way
for any pair $ x, y $.
In order to properly choose $ i(x), t(x) $, we must appeal to properties of $ X $ being a compact metric space, namely that $ X $ is separable.

\begin{thm}
  $ CICT $ is closed.
\end{thm}

\begin{proof}
 Let $ A \in \overline{CICT} $ and choose a countable, dense subset of $ A $, $ \Lambda = \left\{ x_{1}, x_{2}, \dots \right\} $.
 Choose a sequence $ \left\{ B_{i} \right\}_{i \in \mathbb{N}} \subseteq CICT $ so that $ d_{H}(A, B_{n}) < \frac{1}{n} $.
 For each $ x_{n} $ and $ k \in \mathbb{N} $ choose $ x_{n}^{k} \in B_{k} $ with $ d(x_{n}, x_{n}^{k}) < \frac{1}{k} $.
 For $ x_{1} $ we can choose $ i(x_{1}), j(x_{1}) $ such that $ i(x_{1}) = i(x_{1}^{k}) $ and $ j(x_{1}) = j(x_{1}^{k}) $ for
 infinitely many $ k $. By passing to a subsequence $ \left\{ k_{m}(1) \right\}$ of $ \mathbb{N}$, we can assume $ i(x_{1}) = i(x_{1}^{k_{m}(1)}) $ and $ j(x_{1}) = j(x_{1}^{k_{m}(1)}) $ for all $ k_{m}(1)$ in the sequence.
 By induction, we can choose $ i(x_{n}), t(x_{n})$ and a subsequence $ \left\{ k_{m}(n) \right\}$ of $ \left\{ k_{m}(n-1) \right\}$ so that $ i(x_{n}) = i(x_{n}^{k_{m}(n)}) $ and $ j(x_{n}) = j(x_{n}^{k_{m}(n)}) $.
 For $ x \in A \backslash \Lambda $, choose a sequence $ \left\{ x_{n_k} \right\} \subseteq \Lambda $ that converges to $ x $.
 Choose $ i(x), j(x) $ so that $ i(x) = i(x_{n_{k}}) $ and $ j(x) = j(x_{n_{k}}) $ for infinitely many $ k $.

 Let $ x, y \in A $ and $ \epsilon > 0 $ be given. By the uniform continuity of the maps $ f_{i} $ there is a $ \frac{\epsilon}{3} > \delta > 0 $ such that
 $ d(p,q) < \delta $ implies $ d(f_{i}(p), f_{i}(q)) < \frac{\epsilon}{3} $.
 Choose $ x_{n}, x_{l} \in \Lambda $ with $ d(x, x_{n}), d(y, x_{l}) < \frac{\delta}{2} $, and $ i(x_{n}) = i(x) $,
 $ j(x_{k}) = j(y) $. Without loss of generality, assume $ l > n $. Choose $ k_{m}(l) > \frac{2}{\delta}$. By our choice, $i(x_{n}) = i(x_{n}^{k_{m}(l)})$, $j(x_{l}) = j(x_{l}^{k_{m}(l)})$, $ d(x,x_{n}^{k_{m}(l)}), d(y,x_{l}^{k_{m}(l)}) < \delta$, and there
 is a $ \delta $-chain in $ B_{k_{m}(l)}$ from $ x_{n}^{k_{m}(l)} $ to $ x_{l}^{k_{m}(l)} $
 indexed by a word beginning with $ i(x_{n}^{k_{m}(l)})$ and ending with $j(x_{l}^{k_{m}(l)})$.
 Replacing the first $ x_{n}^{k_{m}(l)} $ with $ x $ and the last $ x_{l}^{k_{m}(l)} $ with $ y $, and every other element of the chain with an element of $ A $ within $ \delta$
  gives an $ \epsilon $-chain from $ x $ to $ y $ indexed by
 a word beginning with $ i(x) $ and ending with $ t(y) $. Thus $ A \in CICT $.
\end{proof}

\begin{lem}
	$ \mathfrak{W}_{w} \subseteq CICT$.
\end{lem}

\begin{proof}
	Fix $ w \in W_{\infty} $ and $ x \in X$.
	Choose $ \frac{1}{3n} > \delta_{n} > 0 $ so that for $ d(y,z) < \delta $, $ d(f_{i}(y), f_{i}(z) ) < \frac{1}{3n} $ and
	an increasing sequence of integers $ \{N_{n}\} $ so that if $ m > N_{n} $ $ d(f_{w_{1} \dots w_{m}}(x), \omega_{w}(x)) < \delta_{n} $.

	For $ y \in \omega_{w}(x)$, choose an increasing sequence of integers $ \left\{k_{n}(y)\right\} $ with \newline
	$ d(f_{w_{1}\dots w_{k_{n}(y)}}(x), y ) < \delta_{n}  $.
	By passing to a subsequence if necessary, we can assume
	$ w_{k_{i}(y)} = w_{k_{i+1}(y)} $ and $ w_{1+k_{i}(y)} = w_{1+k_{i+1}(y)} $ for $ i \in \mathbb{N} $. Set $ t(y) = w_{k_{i}(y)} $ and $ i(y) = w_{1+k_{i}(y)} $. Note that $ i(y) \neq t(y)^{-1} $ as both are consecutive letters in a reduced word.

	For given $ y, z \in \omega_{w}(x)$ and $ \epsilon > 0 $, find $ \frac{1}{n} < \epsilon $. Choose $ N_{n} < k_{l}(y) < k_{m}(z) $.
	Just as in the proof that $ \omega_{w}(x) $ is ICT we can find
	a $ \epsilon $-chain in $ \omega_{w}(x) $ between $ y $ and $ z $ indexed by $ w_{1+k_{l}(y)} \dots w_{k_{m}(z)} $.
\end{proof}

In shifts of finite type $X$, for a given $Y\in CICT$,  we can explicitly construct a word $w\in W_\infty$ and $x\in X$ with $Y=\omega_w(x)$.
The following lemma helps to verify our construction.

\begin{lem}\label{shiftPseudoOrbit}
	Let $ X $ be a shift space and fix a function $ \mathcal{O}: G \rightarrow X $. Suppose $ u \in G $ and $ m \in \mathbb{N} $ with the property
	that for $ v \in \Sigma^{m-1} $ and $ i \in S $, we have \[ d(\sigma_{i}(\mathcal{O}(uv)), \mathcal{O}(uvi)) < 2^{-m} .\] Then $\mathcal{O}(u)(v) = \mathcal{O}(uv)(e) $.
\end{lem}

\begin{proof}
	Fix $ v = v_{0} \dots v_{n} \in \Sigma^{m-1} $. As $ d( \sigma_{v_{0}}(\mathcal{O}(u)), \mathcal{O}(uv_{0}) ) < 2^{-m} $,
	the uniform continuity of the shift maps gives
	$ d( \sigma_{v_{0}v_{1}}(\mathcal{O}(u)), \sigma_{v_{1}}(\mathcal{O}(uv_{0})) ) < 2^{-m + 1} $.
	Since  $ d( \sigma_{v_{1}}(\mathcal{O}(uv_{0})), \mathcal{O}(uv_{0}v_{1})) < 2^{-m} $
	we have $ d( \sigma_{v_{0}v_{1}}(\mathcal{O}(u)), \mathcal{O}(uv_{0}v_{1})) < 2^{-m + 1} $.
	By induction we see $ d( \sigma_{v_{0}\dots v_{i}}(\mathcal{O}(u)), \mathcal{O}(uv_{0}\dots v_{i})) < 2^{-m + i} $.

	Therefore $ d( \sigma_{v_{0}\dots v_{n}}(\mathcal{O}(u)), \mathcal{O}(uv_{0}\dots v_{n})) < 2^{-m + n} \leq 2^{-1} $.
	This implies that  $ \sigma_{v}(\mathcal{O}(u))(e)  = \mathcal{O}(uv)(e) $.
	The left-hand side of this equation is equal to $ \mathcal{O}(u)(v)$, so our lemma holds.
\end{proof}

\begin{thm}\label{CICT}
	If $ X $ is a shift of finite type and $ Y \in CICT$, then $ Y = \omega_{w}(x)$ for some
	$ x \in X $ and $ w \in W_{\infty} $.
\end{thm}

\begin{proof}
	Let $ X $ be $ M $-step.
	For $ k \geq M + 1 $, let $ \left\{x^{k}_{i}\right\}_{i=0}^{n_{k}} \subseteq Y $ be sequence that $ 2^{-k} $ covers $Y$.
	As $ Y \in CICT $, there is a $ 2^{-k} $-chain from $x^{k}_{i} $ to $ x^{k}_{i+1} $ indexed by $ u_{i} $ where $ u_{i} $
	begins with $ i(x^{k}_{i}) $ and ends with $ t(x^{k}_{i+1}) $.
	By concatenating these chains, we can get a $ 2^{-k} $-chain $ \left\{y^{k}_{0}, \dots, y^{k}_{n_{k}} \right\} $ from $ x^{k}_{0} $ to $ x^{k}_{n_{k}} $ indexed by
	$ w_{k} = v^{k}_{1} \dots v^{k}_{m_{k}} $ such that $ v^{k}_{1} = i(x^{k}_{0}) $, $ v^{k}_{m_{k}} = t(x^{k}_{n_{k}}) $, and for every $ i $ there is an $ n $ such that $ x^{k}_{i} = y^{k}_{n} $.
	We also have a $ 2^{-k-1} $-chain $ \left\{z^{k}_{0}, \dots, z^{k}_{l_{k}} \right\} $ from $ x^{k}_{n_{k}} $ to $ x^{k+1}_{0} $ indexed by $ w'_{k} = v'_{1} \dots v'_{l_{k}} $ with $ v'_{1}  = i(x^{k}_{n_{k}}) $ and $ v'_{l_{k}} = i(x^{k+1}_{0}) $.

	Concatenating $ w = w_{M+1}w'_{M+1}w_{M+2}w'_{M+2} \dots $ and \newline $ \left\{y^{M+1}_{0}, \dots, y^{M+1}_{n_{M+1}}, z^{M+1}_{1}, \dots, z^{M+1}_{l_{M+1}}, y^{M+2}_{1}, \dots, y^{M+2}_{n_{M+2}}, \dots \right\} $,
	yields a sequence \newline $ \left\{z_{0}, z_{1}, \dots \right\} \subseteq Y $ and $ w = t_{1}t_{2} \dots \in W_{\infty} $ such
	that for all $ i \in \mathbb{N} $ $ d(\sigma_{t_{i+1}}(z_{i}), z_{i+1}) < 2^{-M-1} $ and for $ n > M+1 $ there is $ k_{n} $ such that for
	$ m > k_{n} $, $ d(\sigma_{t_{m}+1}(z_{m}), z_{m+1}) < 2^{-n} $.

	In order to keep track of these points and their interrelations, we define a function $ \mathcal{O}: G \rightarrow Y $. Set $ \mathcal{O}(t_{1}\dots t_{m}) = z_{m} $ and $ \mathcal{O}(e) = z_{0} $. For notation, let $ t_{0} = e $. For all other $ v \in G $
	let $ n_{v} $ be the largest integer such that $ t_{0} \dots t_{n_{v}} $ is a prefix of $ v = t_{0} \dots t_{n_{v}} v' $.
	Then let $ \mathcal{O}(v) = \sigma_{v'}(z_{n_{v}}) $.

	By our construction, $ \mathcal{O} $ has the properties that:
	\begin{enumerate}
		\item For $ u \in G $, $ v \in \Sigma^{M} $, and $ i \in S $, $ d(\sigma_{i}(\mathcal{O}(uv)), \mathcal{O}(uvi)) < 2^{-M-1} $.
		\item For every $ n \in N $ there is $ k_{n} \in \mathbb{N} $ such that for $ \left| u \right| > k_{n} $, $ v \in \Sigma^{n} $, and
		$ i \in R $, $ d(\sigma_{i}(\mathcal{O}(uv)), \mathcal{O}(uvi)) < 2^{-n-1} $.
	\end{enumerate}

	Define $ x: G \rightarrow \mathcal{A} $ by $ x(u) = \mathcal{O}(u)(e) $.
	We claim that $ x \in X $ and that $ \omega_{w}(x) = A $.
	By property (1) and Lemma \ref{shiftPseudoOrbit}, every $ M $-block in $ x $ is the central $M$-block for some $ \mathcal{O}(v) $.
	As every $ \mathcal{O}(v) \in X $, this implies $ x \in X $.
	Property (2) implies that $ d(\sigma(x), \mathcal{O}(v)) < 2^{-n} $ for $ \left| v \right| > k_{n} $.

	Let $ y \in Y $ and $ n > \mathbb{N} $.
	Find $ k, i, m $ such that $ d(y, x^{k}_{i}) < 2^{-n-1} $, $ \mathcal{O}(u_{1} \dots u_{m}) = x^{k}_{i} $, with $ m > k_{n+1} $.
	Therefore $ d(y, \sigma_{u_{1} \dots u_{m}}(x)) < 2^{-n} $, so $ y \in \omega_{w}(x) $.

	Now suppose $ y \in \omega_{w}(x) $. For $ n \in \mathbb{N} $ find $ m > k_{n+1} $ so that $ d(y, \sigma_{u_{1} \dots u_{m}}(x)) < 2^{-n-1}$ and $ d(\sigma_{u_{1} \dots u_{m}}(x), \mathcal{O}(u_{1} \dots u_{m})) < 2^{-n-1} $.
	Therefore $ d(y,  \mathcal{O}(u_{1} \dots u_{m})) < 2^{-n} $. As $ \mathcal{O}(u_{1} \dots u_{m} ) \in Y $ by construction
	and $ Y $ is closed, $ y \in Y $.

	Therefore $ Y  = \omega_{w}(x) $.
\end{proof}

\begin{cor}
	Let $X$ be a shift of finite type over $G$ . Then $\mathfrak W_w = CICT$. If $G$ is a free monoid, then $\mathfrak W_w= ICT$.
\end{cor}

\begin{definition}
	Given a metric space $ X $ and $ \delta > 0 $, a \emph{$ \delta $-G-pseudo orbit} is a function $ \mathcal{O}: G \rightarrow ֤֤֤X $ such that for $ u \in G $ and $ i \in S $ we have $ d(f_{i}(\mathcal{O}(u)), \mathcal{O}(ui) ) < \delta $.
\end{definition}

\begin{definition}
	Let $Y$ be a closed subset of $X$. $ Y $ is \emph{internally block transitive} ($Y \in IBT$)
	if for every $ \delta > 0 $ $ x_{1}, \dots , x_{n} \in Y $ there is a $ \delta $-G-pseudo orbit $ \mathcal{O} $ of $ Y $ containing each $ x_{i} $.
\end{definition}

The notion of internally block transitive is related to that of mesh transitivity explored in \cite{MR-Zd}. The choice of the word block in this terminology might seem spurious, but it is worth noting that the conditions on the pseudo-orbit are finite in nature and in the context of a free group or monoid, a $ \delta$-G-pseudo orbit can be easily generated by a extending a block function $ B: \Sigma_{n} \to Y $  with
$ d(B(ui), \sigma_{i}(B(u))< \delta$ for $ u \in G $ and $ i \in S $.

\begin{thm}\label{IBTClosed}
 For any $ X $, $ IBT $ is closed.
\end{thm}

\begin{proof}
 Suppose that $ Y \in \overline{IBT} $ and $ \epsilon > 0 $. By uniform continuity, choose $ \frac{\epsilon}{3} > \delta > 0 $
 so the $ d(p,q) < \delta $ implies $ d(f_{i}(p), f_{i}(q)) < \frac{\epsilon}{3} $.
 Choose $ x_{1}, \dots , x_{n} \in Y $ and $ B \in IBT $ with $ d_{H}(Y, B) < \delta $.
 Find $ y_{i} \in B $ with $ d(x_{i}, y_{i}) < \delta $ and let $ \mathcal{O} $ be a $ \delta $-G-pseudo orbit of $ B $ containing $ y_{1}, \dots , y_{n} $.
 We create a pseudo orbit $ \mathcal{O}' $ in $ A $ by replacing every $ y_{i} $ in $ \mathcal{O} $ with $ x_{i} $ and
 replacing any other $ y \in B $ with $ x \in A $ so that $ d(x,y) < \delta $.
 By the choice of $ \delta $, it is easy to see that $ \mathcal{O}'$ is an $ \epsilon $-G-pseudo orbit containing  $ x_{0}, \dots, x_{n} $.
 Therefore $ A \in IBT $ and $ IBT $ is closed.
\end{proof}

\begin{thm}
 If $ Y \in IBT $, $ Y $ is invariant.
\end{thm}

\begin{proof}
 Let $ y \in Y $ and consider $ f_{i}(y) $. For $ n \in \mathbb{N} $ there is
 a $ \frac{1}{n} $-G-pseudo orbit $ \mathcal{O}_{n} $ containing $ y $. Say $ \mathcal{O}_{n}(u_{n}) = y $ for all $ n $.
 Then $ d(f_{i}(\mathcal{O}_{n}(u_{n})), \mathcal{O}_{n}(u_{n}i) ) < \frac{1}{n}$.
 Hence $ d(f_{i}(y), \mathcal{O}_{n}(u_{n}i) ) < \frac{1}{n} $.
 As $ \mathcal{O}_{n}(u_{n}i) \in Y $ and $ Y $ is closed, $ f_{i}(y) \in Y $.
\end{proof}

With this definition, it is not guaranteed that we can concatenate these $ \delta $-G-pseudo orbits together in a meaningful way.
The following definitions will allow us sufficient conditions for which we have meaningful concatenation. The first is most applicable
to free group actions, the second to free monoid actions.

\begin{definition}
	If $ Y\in IBT$, an element $ y \in Y $ is \emph{i,j final}  if for every $ \delta > 0 $, $ x_{1}, \dots, x_{n} \in Y $ there is
	a $ \delta $-G-pseudo orbit $ \mathcal{O} $ of $ Y $ and indexes $ u_{k} $ such that $ \mathcal{O}(u_{k}) = x_{k} $ and $ u_{k} \neq u_{m} $ for $ k \neq m $
	and indexes $ u_{i}, u_{j} $ with $ \mathcal{O}(u_{i}) = \mathcal{O}(u_{j}) = y $, $ u_{i}, u_{j} $ end in $ i \neq j $ respectively and $ u_{i}, u_{j} $ are not
	prefixes of each other or any $ u_{k} $.
\end{definition}

\begin{definition}
	Let $Y$ be a closed subset of $X$. $Y\in IBT^*$ if and only if $Y\in IBT $ and there exists $y\in Y$ that is $i,j$-final.
\end{definition}

\begin{definition}
	If $ Y \in IBT $, an element $ y \in Y $ is \emph{final} if for every $ \delta > 0 $, $ x_{1}, \dots, x_{n} \in Y $ there is
	a $ \delta $-G-pseudo orbit $ \mathcal{O} $ of $ Y $ and indexes $ u_{k} $ and $ u_{y} $ such that $ \mathcal{O}(u_{k}) = x_{k} $, $ \mathcal{O}(e) = \mathcal{O}(u_{y}) = y $ with $ u_{y} $ not a prefix of any $ u_{k} $.
\end{definition}

\begin{definition}
	Let $Y$ be a closed subset of $X$. $Y\in IBT^{\circ}$ if and only if $Y\in IBT $ and there exists $y\in Y$ that is final.
\end{definition}

\begin{thm}\label{IBT*closed}
 $ IBT^{*} $ is closed.
\end{thm}

\begin{proof}
 Let $ Y \in \overline{IBT^{*}} $ and let $ \left\{ B_{n} \right\}_{n \in \mathbb{N}} \subseteq IBT^{*} $ converge to $ Y $.
 In each $ B_{n} $ there a point $ x_{n} $ that is $ i_{n}, j_{n} $ final.
 Choose $ i, j $ so that $ i = i_{n}, j= j_{n} $ infinitely often. By passing to a subsequence if necessary, we can assume
 $ i = i_{n}, j = j_{n} $ for all $ n $. Choose $ x \in X $ so $ \left\{ x_{n} \right\}_{n \in \mathbb{N}} $ converges to
 $ x $. By a similar technique to Theorem \ref{IBTClosed}, it is not hard to see $ Y \in IBT^{*} $ with $ x $ being $ i,j $ final.
\end{proof}

\begin{thm}
	 $ IBT^{\circ} $ is closed.
\end{thm}

\begin{proof}
	This follows from the technique used in Theorem \ref{IBTClosed}.
\end{proof}

\begin{lem}
	Let $ w \in W_{\infty} $, $ x \in X $ be given. For a free group action, $ \omega_{F_{w}}(x) \in IBT^{*}$ with $ y \in \omega_{w}(x) $ i,j-final for some $ i, j $.
\end{lem}

\begin{proof}
	Let $ x_{1}, \dots, x_{n} \in \omega_{F_{w}}(x) $, $ y \in \omega_{F_{w}}(x) $, and $ \delta > 0 $ be given.
	By the uniform continuity of the maps $ f_{i} $, find $ \frac{\delta}{3} > \eta > 0  $ such that if $ d(x,y) < \eta $ then $ d(f_{i}(x), f_{i}(y)) < \frac{\delta}{3} $
	for $ i \in S $.
	As $ y \in \omega_{w}(x) $ and $ \omega_{w}(x) $ is CICT, let $ i = i(y)^{-1} $ and $ j = t(y) $, so $ i \neq j $.

	Now we can find $ w_{1} \dots w_{k_{i}}w_{k_{i}+1} \in F $ with $ i = w_{k_{i}+1}^{-1} $, $ d(y, f_{w_{1} \dots w_{k_{i}}}(x)) < \eta $
	and $ k_{i} > N_{\eta} $ given in Lemma \ref{omegaFSufClose}.
	For $ x_{1} $ we can find $ w_{1} \dots w_{k_{1}}v_{x_{1}} \in F $ with $ k_{1} > k_{i} + 1 $ such that $ d(x_{1}, f_{w_{1} \dots w_{k_{1}}v_{x_{1}}}(x)) < \eta $ and
	$ v_{x_{1}} $ does not begin with $ w_{k_{1} + 1} $.
	By induction we can find $ k_{m+1} > k_{m} $, $ v_{x_{m+1}} \in F $ such that $ d(x_{m+1}, f_{w_{1} \dots w_{k_{m+1}}v_{x_{m+1}}}(x)) < \eta $
	and $ v_{x_{m+1}} $ does not begin with $ w_{k_{m+1}+1} $. Then we can get
	$ k_{j} > k_{n} $ with $ w_{k_{j}} = j $ and $ d(y, f_{w_{1} \dots w_{k_{j}}}(x)$.

	Now define $ \mathcal{O}: F \rightarrow \omega_{F_{w}}(x) $ by $ \mathcal{O}(w_{k_{i}+ 2} \dots w_{k_{m}}v_{x_{m}}) = x_{m} $,
	$ \mathcal{O}(w_{k_{i}+ 1}^{-1}) = \mathcal{O}(w_{k_{i}+ 2} \dots w_{k_{j}}) = y $.
	For $ u = w_{k_{i}+ 1}^{-1}v $, define $ \mathcal{O}(u) = f_{v}(y) $ and for all other $ u $ that has not yet been defined, we
	can choose $ z \in \omega_{w}(x) $ with $ d(z, f_{w_{1}\dots w_{k_{i}+1}u}(x)) < \eta $ and let $ \mathcal{O}(u) = z $.
 By construction $ \mathcal{O} $ is an $ \delta $-pseudo orbit.

	For the indexes, let $ u_{i} = w_{k_{i} + 1}^{-1} $, $ u_{j} = w_{k_{i} + 2} \dots w_{k_{j}} $, and $ u_{m} = w_{k_{i} + 2} \dots w_{k_{m}}v_{x_{m}} $.
	By construction $ \mathcal{O}(u_{m}) = x_{m} $ and $ u_{k} \neq u_{m} $ for $ k \neq m $,
	$ \mathcal{O}(u_{i}) = \mathcal{O}(u_{j}) = y $, $ u_{i}, u_{j} $ end in $ i, j $ respectively and $ u_{i}, u_{j} $ are not
	prefixes of each other or any $ u_{m} $.
\end{proof}

\begin{lem}
	Let $ w \in W_{\infty} $, $ x \in X $ be given. For a free monoid action, $ \omega_{F_{w}}(x) \in IBT^{\circ}$ with $ y \in \omega_{w}(x) $ final.
\end{lem}

\begin{proof}
	Let $ x_{1}, \dots, x_{n} \in \omega_{F_{w}}(x) $, $ y \in \omega_{F_{w}}(x) $, and $ \delta > 0 $ be given.
	By the uniform continuity of the maps $ f_{i} $, find $ \frac{\delta}{3} > \eta > 0  $ such that if $ d(x,y) < \eta $ then $ d(f_{i}(x), f_{i}(y)) < \frac{\delta}{3} $
	for $ i \in S $.

	Now we can find $ w_{1} \dots w_{k} $ such that $ d(y, f_{w_{1} \dots w_{k}}(x)) < \eta $.
	Find $ k_{1} > k > N_{\eta}$ from Lemma \ref{omegaFSufClose} and $ u^{1} = u_{1} \dots u_{m} $ with $ u_{1} \neq w_{k_{1}+1} $ and $ d(x_{1}, f_{w_{1} \dots w_{k_{1}}u^{1}}(x)) < \eta$.
	Inductively find $ k_{i} > k_{i-1} $ and $ u^{k} = u_{1} \dots u_{m} $ with $ u_{1} \neq w_{k_{i}+1} $ and $ d(x_{i}, f_{w_{1} \dots w_{k_{i}}u^{k}}(x)) < \eta $. Finally, find $ k_{y} > k_{n} $ and $ u^{y} = u_{1} \dots u_{m} $ with $ u_{1} \neq w_{k_{n}+1} $ and
	$ d(y, f_{w_{1}\dots w_{k_{y}}u^{y}}, y) < \eta $.
 Define $ \mathcal{O}: H \rightarrow \omega_{w}(x) $ by $ \mathcal{O}(e) = \mathcal{O}(u^{y}) = y $, $ \mathcal{O}(u^{k}) = x_{k} $ and for
	all other $ u $ choose $ z \in \omega_{w}(x) $ so that $ d(f_{w_{1} \dots w_{k}u}(x), z) < \eta $.
	This satisfies the requirements for $ IBT^{\circ} $.
\end{proof}

\begin{thm}\label{IBTgroup}
	Suppose $ X \in \mathcal{A}^{F} $ is SFT with largest forbidden block size M, and let $ Y \subseteq X $ be IBT, invariant, and compact with some $ y \in Y $ i,j-final.
	Then $ Y = \omega_{F_{w}}(\bar{x}) $ for some $ w \in W $ and $ \bar{x} \in X $.
\end{thm}

\begin{proof}
	For $ n > M $, choose $ \left\{x_{1}^{n}, \dots, x_{k_{n}}^{n} \right\} \subseteq Y $ a $ 2^{-n} $ cover of $ Y $. By assumption, we can find a $ 2^{-n} $-pseudo orbit $ \mathcal{O}_{n} $
	with indexes $ u_{i}^{n}, u_{j}^{n} \in \Sigma^{k_{n}} $ ending in $ i $ and $ j $ respectively such that $ \mathcal{O}_{n}(u_{i}^{n}) =  \mathcal{O}_{n}(u_{j}^{n}) = y $
	and indexes $ u_{m}^{n} $ such that $ \mathcal{O}_{n}(u_{m}^{n}) = x_{m}^{n} $ and $ u_{i}^{n}, u_{j}^{n} $ are not prefixes of $ u_{m}^{n} $ for all $ m $.

	We will inductively construct a function $ \mathcal{O}: F \rightarrow Y $ and a word $ w \in W_{\infty} $.

	First we construct $ w $. Define $ w_{1} = u_{j}^{M+1} $ and for $ n > 1 $, $ w_{n} = w_{n-1}(u_{i}^{M+n})^{-1}u_{j}^{M+n} $. By induction we will prove $ w_{n} $ ends
	with $ j $ and begins with $ w_{n-1} $. Note $ u_{j}^{M+1} $ ends with $ j $ so $ w_{1} $ ends with $ j $. Suppose for our inductive step that $ w_{n-1} $ ends with $ j $.
	As $ u_{i}^{n} $ ends with $ i $ and is not a prefix of $ u_{j}^{n} $ and $ u_{j}^{n} $ ends with $ j $ we have $ (u_{i}^{n})^{-1}u_{j}^{n} $ beginning with $ i^{-1} $ and ending with $ j $.
	Therefore $ w_{n} $ ends with $ j $. Finally, because $ i \neq j $, $ w_{n} $ begins with $ w_{n-1} $. Define $ w = \lim_{n \rightarrow \infty} w_{n} $.

	For ease in constructing $ \mathcal{O} $, define $ \mathcal{O}'_{n} $ to be $ \mathcal{O}_{n} $ restricted to elements of $ \Sigma^{k_{n}} $ which do not have
	$ u_{i}^{n}, u_{j}^{n} $ as a proper prefix. Let $ \mathcal{D}_{n} $ be the domain of $ \mathcal{O}'_{n} $.
	For $ u \in \mathcal{D}_{M+1} $ define $ \mathcal{O}(u) = \mathcal{O}'_{M+1}(u) $.
	Then for $ u \in \mathcal{D}_{M+n} $ define $ \mathcal{O}(w_{n-1}(u_{i}^{M+n})^{-1}u) = \mathcal{O}_{M+n}'(u) $.
	We will show that this step is well-defined.
	Suppose for some $ n < m $ there is  $ v, v' $ in $ \mathcal{D}_{n}, \mathcal{D}_{m} $ respectively such that $ w_{n-1}(u_{i}^{M+n})^{-1}v = w_{m-1}(u_{i}^{M+m})^{-1}v' $.
	Write $  w_{m-1} = w_{n-1}(u_{i}^{M+n})^{-1}u_{j}^{M+n} \cdots (u_{i}^{M+m-1})^{-1}u_{j}^{M+m-1} $. \newline
	Thus $ (u_{i}^{M+n})^{-1}v = (u_{i}^{M+n})^{-1}u_{j}^{M+n} \cdots (u_{i}^{M+m-1})^{-1}u_{j}^{M+m-1}(u_{i}^{M+m})^{-1}v'$. \newline
	Therefore $ v = u_{j}^{M+n} \cdots (u_{i}^{M+m-1})^{-1}u_{j}^{M+m-1}(u_{i}^{M+m})^{-1}v' $.
	As $ v $ does not contain $ u_{j}^{M+n} $ as a proper prefix, $ v =  u_{j}^{M+n} $, $ m = n + 1 $ and $ v' = u_{i}^{M+n+1} $.
	In this case, $ \mathcal{O}(w_{n-1}(u_{i}^{M+n})^{-1}u_{j}^{M+n}) =  \mathcal{O}_{M+n}'(v) = y $ and
	$  \mathcal{O}(w_{n}(u_{i}^{M+n+1})^{-1}u_{i}^{M+n+1}) = \mathcal{O}_{M+n+1}'(u_{i}^{M+n+1}) = y $. Thus $ \mathcal{O} $, so far as it has been defined is well-defined.

	To complete the construction of $ \mathcal{O} $, for $ u \in F $ with $ \mathcal{O}(u) $ not already defined, let $ u' $ be the largest prefix of $ u $ with $ \mathcal{O}(u') $ already
	defined. Such $ u' $ always exists as $ \mathcal{O}(e) $ is already defined.
	Then define $ \mathcal{O}(u) = \sigma_{u'^{-1}u}(\mathcal{O}(u')) $.
	Thus we have defined $ \mathcal{O}: F \rightarrow Y $.

	From the construction, By our construction, $ \mathcal{O} $ has the properties that:
	\begin{enumerate}
		\item For $ u \in F $, $ v \in \Sigma^{M} $, and $ i \in S $, $ d(\sigma_{i}(\mathcal{O}(uv)), \mathcal{O}(uvi)) < 2^{-M-1} $.
		\item For every $ n \in N $ there is $ k_{n} \in \mathbb{N} $ such that for $ \left| u \right| > k_{n} $, $ v \in \Sigma^{n} $, and
		$ i \in S $, $ d(\sigma_{i}(\mathcal{O}(uv)), \mathcal{O}(uvi)) < 2^{-n-1} $.
	\end{enumerate}
	Just as in Theorem \ref{CICT}, by defining $ \bar{x} $ by $ \bar{x}(v) = \mathcal{O}(v)(e) $ these properties imply that
	$ \bar{x} \in X $ and $ d(\sigma(\bar{x}), \mathcal{O}(v)) < 2^{-n} $ for $ \left| v \right| > k_{n} $.

	Finally, we show that $ Y = \omega_{F_{w}}(\bar{x}) $. First $ \omega_{F_{w}}(\bar{x}) \subseteq Y $ as every N-block in $ \bar{x} $ is an N-block of an element of $ A $.
	Now let $ z \in Y $ and $ n \in \mathbb{N} $. We must find a $ u \in F $ with $ u_{|n} = w_{|n} $ and $ d(\sigma_{u}(\bar{x}), z) < \frac{1}{n} $.
	Note $ \left| w_{n} \right| \geq n $.
	Find $ x_{m}^{M+n+1} $ such that $ d(x_{m}^{M+n}, z) < 2^{-(M+n+1)} $. Let $ u = w_{n}(u_{i}^{M+n})^{-1}u_{m}^{M+n+1} $.
	Thus $ u_{|n} = w_{|n} $ and $ d(\sigma_{u}(\bar{x}), x_{m}^{M+n+1}) < 2^{-(M+n+1)} $. Therefore $  d(\sigma_{u}(\bar{x}), z) < 2^{-(M+n+1)} < \frac{1}{n} $.
	Thus $ Y = \omega_{F_{w}}(\bar{x}) $.
\end{proof}

Note that this construction depended only upon the properties of $ A \in IBT^{*} $. As long as one replaces the shift maps $ \sigma $ with
just a more general function $ f $, the construction works in all other $ IBT^{*} $ sets with $ F $-actions.

\begin{cor}
	Let $X$ be a shift of finite type over a free group $G$. Then $\mathfrak W_{F_w}=IBT^*$.
\end{cor}

	Using a slightly different, yet more straightforward, construction, we obtain the same result for monoid actions.

\begin{thm}\label{IBTmonoid}
	Suppose $ X \in \mathcal{A}^{H} $ is SFT with largest forbidden block size M, and let $ Y \subseteq X $ be IBT, invariant, and compact with some $ y \in Y $ final.
	Then $ Y = \omega_{F_{w}}(\bar{x}) $ for some $ w \in W_{\infty} $ and $ \bar{x} \in X $.
\end{thm}

\begin{proof}
	For $ n >M $, choose $ \left\{x_{1}^{n}, \dots, x_{k_{n}}^{n} \right\} \subseteq Y $ a $ 2^{-n} $ cover of $ Y $. By assumption,
	we can find a $ 2^{-n} $-pseudo orbit $ \mathcal{O}_{n} $
	with indexes $ u_{y}^{n}, u_{m}^{n} $ such that $ \mathcal{O}(e) = \mathcal{O}(u^{n}_{y}) = y$, $ \mathcal{O}_{n}(u_{m}^{n}) = x_{m}^{n} $
	and $ u_{y}^{n} $ is not a prefix of 	$ u_{m}^{n} $ for all $ m $.

	Define $ w = u_{y}^{M+1}u_{y}^{M+2} \dots $. For notation, define $ u_{y}^{0} = e $.
	For $ u \in H $, find the maximal $ m $ such that $ u = u_{y}^{0} \dots u_{y}^{m}u' $ and define $ \mathcal{O}(u) = \mathcal{O}_{m+1}(u') $.
	Letting $ \bar{x}(v) = \mathcal{O}(v)(e) $, the same reasoning as above gives $ \omega_{F_{w}}(\bar{x}) = Y $.
\end{proof}

\begin{cor}
	Let $X$ be a shift of finite type over a free monoid $H$. Then $\mathfrak W_{F_w}=IBT^\circ$.
\end{cor}


\section{Shadowing in Group Actions}

In the previous section, we developed some connections between limit sets and internal transitivity properties in the context of shifts of finite type over finitely generated free groups. It is not surprising, given the results concerning these relationships in $\Z$- and $\N$-actions \cite{MR, BDG-tentmaps, BGKR-omega, GoodMeddaugh-ICT} that we are able to find analogous results outside of shift spaces. In order to demonstrate these properties, we first need to have an appropriate notion of shadowing for group actions.

\begin{definition}
	For $ \epsilon > 0 $ a function $ \mathcal{O}: G \rightarrow ֤֤֤X $ is \emph{$ \epsilon $-shadowed by $x \in X $}
	if $ d(f_{u}(x), \mathcal{O}(u)) < \epsilon $ for all $ u \in G $.
\end{definition}

\begin{definition}
	A $G$-action on a compact metric space $X$ has the \emph{G-shadowing property} if for every $ \epsilon > 0 $ there exists $ \delta > 0 $ such that every $ \delta $-$G$-pseudo orbit
	is $ \epsilon $-shadowed by a point in $ X $.
\end{definition}

\begin{lem}
	If $ X $ is not an SFT, then for every $ n \in \mathbb{N} $ there is $ m > n $ such that there is a forbidden m-block of $ X $ such that
	every sub-block is not forbidden.
\end{lem}

\begin{proof}
	We will prove the contrapositive. Write $ X = X_{\mathcal{F}} $ for some set of forbidden blocks $ \mathcal{F} $.
	Let $ m $ be the largest integer such that there is an m-block $ B \in \mathcal{F} $ such that every sub-block of
	$ B $ is not in $ \mathcal{F} $. Let $ \mathcal{F}' = \left\{B \in \mathcal{F} : B \; \text{is a k-block for } \; k \leq m \right\} $.
	Note that $ \mathcal{F}' $ is finite as there only finitely many k-blocks for $ k \leq m $.
	We claim that $ X_{\mathcal{F}} = X_{\mathcal{F}'} $.
	As $ \mathcal{F}' \subseteq \mathcal{F} $, $ X_{\mathcal{F}'} \supseteq X_{\mathcal{F}} $.
	Now suppose $ x \notin X_{\mathcal{F}} $. Thus $ x $ contains a forbidden k-block $ B $ in $ \mathcal{F} $.
	If $ k \leq m $, $  B \in \mathcal{F}' $ so $ x \notin X_{\mathcal{F}'} $.
	If $ k > m $ then $ B $ contains a forbidden l-block for $ l < k $. By induction, $ B $ contains a forbidden l-block for $ l \leq m $.
	In either case, $ x $ contains a block forbidden in $ \mathcal{F}' $ so $ x \notin X_{\mathcal{F}'} $.
	As $ X = X_{\mathcal{F}} = X_{\mathcal{F}'} $, $ X $ is SFT.
\end{proof}

\begin{lem}
	Suppose $ \mathcal{O} $ is a $ 2^{-m} $-pseudo orbit. Then for $ u \in F $ and $ v \in \Sigma^{m-1} $, $ \mathcal{O}(u)(v) = \mathcal{O}(uv)(e) $.
\end{lem}

\begin{proof}
	This result follows from the same argument as Lemma \ref{shiftPseudoOrbit}.
\end{proof}

\begin{thm}
	A shift space $ X $ is an SFT if and only if $ X $ has the shadowing property.
\end{thm}

\begin{proof}
	Suppose $ X $ is an M-step SFT and let $ \epsilon > 0 $ be given. Choose $ k > M $ so $ 2^{-k} < \epsilon $.
	We claim every $ 2^{-k - 1} $-pseudo orbit can be $ \epsilon $-shadowed. Let $ \mathcal{O} $ be such a pseudo orbit.
	Construct $ x \in \mathcal{A}^{F} $ by $ x(u) = \mathcal{O}(u)(e) $.
	Let $ u \in F $ and $ v \in \Sigma^{k} $.
	By definition, $ \sigma_{u}(x)(v) = x(uv) = \mathcal{O}(uv)(1) $. By the previous lemma, the right-hand term equals $ \mathcal{O}(u)(v) $.
	Thus $ \sigma_{u}(x)_{|\Sigma^{k}} = \mathcal{O}(u)_{|\Sigma^{k}} $ so $ d(\sigma_{u}(x), \mathcal{O}(u)) < 2^{-k} $.
	This implies $ x $ $ \epsilon $-shadows $ \mathcal{O} $.Furthermore, $ x \in X $ as every M-block in $ x $ is an M-block in an element
	of $ X $, thus $ x $ contains no forbidden blocks.

	Now suppose $ X $ is not an SFT. Let $ \epsilon = 2^{-1} $ and suppose $ X $ has the shadowing property.
	Thus there is a $ \delta > 0 $ such that every $ \delta $-pseudo orbit can be $ \epsilon $-shadowed.
	Choose $ m $ with $ 2^{-m} < \delta $.
	By a previous lemma, there is a $ k > m + 2 $ such that $ X $ has a forbidden k-block $ B $ with all sub-blocks of $ B $ not forbidden.
	For $ i \in S $ let $ B_{i} $ be the (k-1)-block of $ B $ centered at $ i $.
	As these are not forbidden, there exists an $ x_{i} \in X $ such that $ x_{i |\Sigma^{k-1}} = B_{i} $.
	Let $ B_{1}  = B_{|\Sigma^{k-1}} $ and $ x_{1} \in X $ such that $ x_{1|\Sigma^{k-1}} = B_{1} $.
	For $ u \in F $, define $ \mathcal{O}(iu) = \sigma_{u}(x_{i}) $ and $ \mathcal{O}(e) = x_{1} $.
	For $ i \in S $, $ d(\sigma_{i^{-1}}(x_{i}), x_{1}) < 2^{-k + 2} $ as $ B_{i} $ and $ B_{0} $ overlap on a (k-2)-block.
	For $ i,j \in S $ and $ u \in F $ with $ iuj \neq 1 $, $ \mathcal{O}(iuj) = \sigma_{uj}(x_{i}) = \sigma_{j}(\sigma_{u}(x_{i})) =
	\sigma_{j}(\mathcal{O}(iu)) $. Therefore $ \mathcal{O} $ is a $ \delta $-pseudo orbit.
	Suppose that $ x \in X $ $ \epsilon $-shadows $ \mathcal{O} $.
	Then for $ u \in F $, $ d(\sigma_{u}(x), \mathcal{O}(u) )  < \epsilon $.
	Particularly, this implies $ x(u) = \mathcal{O}(u)(e) $.
	We claim that $ x $ contains $ B $.
	Clearly, $ x(e) = B(e) $. For $ i \in S $ and $ u \in \Sigma^{k-1} $,
	$ x(iu) = \sigma_{iu}(x)(e) = \mathcal{O}(iu)(e) = \sigma_{u}(x_{i})(e) = x_{i}(u) = \mathcal{O}(i)(u) = B_{i}(u) = B(iu) $.
	Therefore our claim is correct, hence $ \mathcal{O} $ cannot be $ \epsilon $-shadowed, contradicting the assumption of
	$ X $ having the shadowing property.
\end{proof}

\begin{thm}\label{CICTgeneral}
	For an $G$-action on $X$ with $G$-shadowing, then $CICT=\overline{\mathfrak{W}_w}$.
\end{thm}

\begin{proof}
 Let $ Y \in CICT $. For $ n \in \mathbb{N} $, find $ \delta_{n} $ such that any $ \delta_{n} $-pseudo orbit can $ \frac{1}{n} $ shadowed.
	Define $ \mathcal{O} $ as follows.
	Fix $ k_{0} > \frac{1}{\delta_{n}} $. For $ k > k_{0} $ let $ \left\{x^{k}_{i}\right\}_{i=0}^{n_{k}} \subseteq Y $ be sequence that $ \frac{1}{k} $ covers $Y$.
	As $ Y $ is CICT, there is a $ \frac{1}{k} $-chain from $x^{k}_{i} $ to $ x^{k}_{i+1} $ indexed by $ u_{i} $ where $ u_{i} $
	begins with $ i(x^{k}_{i}) $ and ends with $ t(x^{k}_{i+1}) $.
	By concatenating these chains, we can get a $ \frac{1}{k} $-chain $ \left\{y^{k}_{0}, \dots, y^{k}_{n_{k}} \right\} $ from $ x^{k}_{0} $ to $ x^{k}_{n_{k}} $ indexed by
	$ w_{k} = v^{k}_{1} \dots v^{k}_{m_{k}} $ such that $ v^{k}_{1} = i(x^{k}_{0}) $, $ v^{k}_{m_{k}} = t(x^{k}_{n_{k}}) $, and for every $ i $ there is an $ n $ such that $ x^{k}_{i} = y^{k}_{n} $.
	We also have a $ \frac{1}{k+1} $-chain $ \left\{z^{k}_{0}, \dots, z^{k}_{l_{k}} \right\} $ from $ x^{k}_{n_{k}} $ to $ x^{k+1}_{0} $ indexed by $ w'_{k} = v'_{1} \dots v'_{l_{k}} $ with $ v'_{1}  = i(x^{k}_{n_{k}}) $ and $ v'_{l_{k}} = i(x^{k+1}_{0}) $.

	Concatenating $ w = w_{k_{0}}w'_{k_{0}}w_{k_{0}+1}w'_{k_{0} + 1} \dots $ and \newline $ \left\{y^{k_{0}}_{0}, \dots, y^{k_{0}}_{n_{k_{0}}}, z^{k_{0}}_{1}, \dots, z^{k_{0}}_{l_{k_{0}}}, y^{k_{0}+1}_{1}, \dots, y^{k_{0}+1}_{n_{k_{0}+1}}, \dots \right\} $,
	yields a sequence \newline $ \left\{z_{0}, z_{1}, \dots \right\} \subseteq Y $ and $ w = t_{1}t_{2} \dots \in W_{\infty} $ such
	that for all $ i \in \mathbb{N} $ $ d(f_{t_{i+1}}(z_{i}), z_{i+1}) < \delta_{n} $ and for $ n > \frac{1}{\delta_{n}} $ there is $ k_{n} $ such that for
	$ m > k_{n} $, $ d(f_{t_{m}+1}(z_{m}), z_{m+1}) < \frac{1}{n} $.

	To construct the pseudo orbit, set $ \mathcal{O}(t_{1}\dots t_{m}) = z_{m} $ and $ \mathcal{O}(e) = z_{0} $. For notation, let $ t_{0} = 1 $. For all other $ v \in F $
	let $ n_{v} $ be the largest integer such that $ t_{0} \dots t_{n_{v}} $ is a prefix of $ v = t_{0} \dots t_{n_{v}}v'$.
	Then let $ \mathcal{O}(v) = f_{v'}(z_{m}) $.

	Let $ x $ be a point that shadows this pseudo-orbit. We want to show that
	$ d_{H}(\omega_{w}(x), Y) < \frac{1}{n} $. For any element of $ y \in \omega_{w}(x) $, $ d(y, Y) < \frac{1}{n} $ by the nature of the construction.
	If $ a \in Y $, then there is a prefix $ v_{m}$ of $ w_{n} $ with $ d(a, \mathcal(O)(v_{m}) ) < \frac{1}{m} $.
	By definition, $ \left\{ f_{v_{m}}(x) \right\} $ converges to a point $ z \in \omega_{w}(x) $.
	Thus $ d(a, z) < \frac{1}{n} $. Therefore $ d_{H}(\omega_{w}(x_{n}), Y) < \frac{1}{n} $. As $ n $ was arbitrary, $ Y \in \overline{\mathfrak{W}_w}$
	and $ CICT \subseteq \overline{\mathfrak{W}_w} $.
	We have already shown that $ \mathfrak{W}_w \subseteq CICT $ and $ CICT $ is closed; thus $ \overline{\mathfrak{W}_w} \subseteq CICT $.
\end{proof}

\begin{figure}
	\centering

	\begin{tikzpicture}[scale = .2]
	\draw (-8,6) -- (-8,4) -- (-8,0) -- (-8,-4) -- (-8,-6);
	\draw (-14,0) -- (-12,0) -- (-8,0) -- (-4,0) -- (-4,2);
	\draw (-10,-4) -- (-6,-4);
	\draw (-12,-2) -- (-12,2);
	\draw (-10,4) -- (-6,4);
	\draw[ultra thick] (-4,-2) -- (-4,0) -- (-2,0);
	\node at (-4,-2) [fill = white]{$y$};
	\node at (-2,0) [fill = white]{$y$};

	\draw (8,6) -- (8,4) -- (8,0) -- (8,-4) -- (8,-6);
	\draw (14,0) -- (12,0) -- (8,0) -- (4,0) -- (4,2);
	\draw (10,-4) -- (6,-4);
	\draw (12,-2) -- (12,2);
	\draw (10,4) -- (6,4);
	\draw (4,-2) -- (4,0) -- (2,0);
	\draw[ultra thick] (4,-2) -- (4,0) -- (8,0) -- (8,4) -- (10,4);
	\node at (4,-2) [fill = white]{$y$};
	\node at (10,4) [fill = white]{$y$};
	\draw[thick, dashed] (-1,0) -- (3,-2);

	\draw[ultra thick] (-2,-25) -- (-2,-16) -- (4,-16) -- (4,-12) -- (7,-12) --
	(7,-10) -- (8, -10);
	\draw (-2,-16) -- (-8,-16) -- (-12,-16) -- (-15,-16);
	\draw (-12, -19) -- (-12, -13);
	\draw (-8,-16) -- (-8,-12) -- (-8,-9);
	\draw (-11,-12) -- (-5,-12);
	\draw (-8,-16) -- (-8,-20) -- (-8,-23);
	\draw (-11,-20) -- (-5,-20);
	\draw (-2,-16) -- (-2,-8);
	\draw (1,-12) -- (4,-12) -- (4,-9);
	\draw (10,-12) -- (9,-12) -- (7,-12) -- (7,-14) -- (7,-15);
	\draw (6,-10) -- (7,-10) -- (7,-9);
	\draw (6,-14) -- (8,-14);
	\draw (9,-13) -- (9,-11);
	\node [below] at (-2,-25) {$ y $};
	\node [below right] at (4,-16) {$ y $};
	\node [above right] at (8,-10) {$ y $};

	\end{tikzpicture}
	\caption{One step of the construction.} \label{fig:M2}

\end{figure}
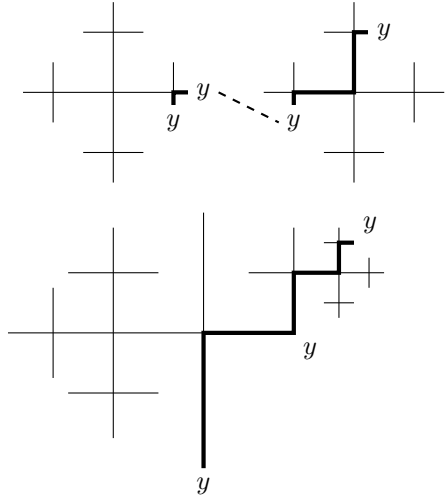

 With a similar construction, we obtain an analogous result for $ IBT $ sets.

\begin{thm}\label{IBTgeneralgroup}
	For an $F$-action on $X$ with $F$-shadowing, $IBT^*=\overline{\mathfrak{W}_{F_w}}$.
\end{thm}

\begin{proof}
	Let $ Y \in IBT^{*} $. For $ k \in \mathbb{N} $, find $ \delta_{k} $ such that any $ \delta_{k} $-pseudo orbit can $ \frac{1}{k} $ shadowed.	Define $ \mathcal{O} $ as follows.
	Fix $ k_{0} > \frac{1}{\delta_{k}} $. For $ k \geq k_{0} $ choose $ \left\{x_{1}^{n}, \dots, x_{k_{n}}^{n} \right\} \subseteq Y $ a $ \frac{1}{n} $ cover of $ Y $. By assumption, we can find a $ \frac{1}{n}$-pseudo orbit $ \mathcal{O}_{n} $
	with indexes $ u_{i}^{n}, u_{j}^{n} \in \Sigma^{k_{n}} $ ending in $ i $ and $ j $ respectively such that $ \mathcal{O}_{n}(u_{i}^{n}) =  \mathcal{O}_{n}(u_{j}^{n}) = y $
	and indexes $ u_{m}^{n} $ such that $ \mathcal{O}_{n}(u_{m}^{n}) = x_{m}^{n} $ and $ u_{i}^{n}, u_{j}^{n} $ are not prefixes of $ u_{m}^{n} $ for all $ m $.

	We will inductively construct $ \mathcal{O}: F \rightarrow Y $ and a word $ w \in W_{\infty} $.

	First we construct $ w $. Define $ w_{0} = u_{j}^{k_{0}} $ and for $ n > 0 $, $ w_{n} = w_{n-1}(u_{i}^{k_{0}+n})^{-1}u_{j}^{k_{0}+n} $. By induction we will prove $ w_{n} $ ends
	with $ j $ and begins with $ w_{n-1} $. Note $ u_{j}^{k_{0}} $ ends with $ j $ so $ w_{0} $ ends with $ j $. Suppose for our inductive step that $ w_{n-1} $ ends with $ j $.
	As $ u_{i}^{n} $ ends with $ i $ and is not a prefix of $ u_{j}^{n} $ and $ u_{j}^{n} $ ends with $ j $ we have $ (u_{i}^{n})^{-1}u_{j}^{n} $ beginning with $ i^{-1} $ and ending with $ j $.
	Therefore $ w_{n} $ ends with $ j $. Finally, because $ i \neq j $, $ w_{n} $ begins with $ w_{n-1} $. Define $ w = \lim_{n \rightarrow \infty} w_{n} $.

	For ease in constructing $ \mathcal{O} $, define $ \mathcal{O}'_{n} $ to be $ \mathcal{O}_{n} $ restricted to elements of $ \Sigma^{k_{n}} $ which do not have
	$ u_{i}^{n}, u_{j}^{n} $ as a proper prefix. Let $ \mathcal{D}_{n} $ be the domain of $ \mathcal{O}'_{n} $.
	For $ u \in \mathcal{D}_{k_{0}} $ define $ \mathcal{O}(u) = \mathcal{O}'_{M}(u) $.
	Then for $ u \in \mathcal{D}_{k-{0}+n} $ define $ \mathcal{O}^{k_{0}}(w_{n-1}(u_{i}^{k_{0}+n})^{-1}u) = \mathcal{O}_{k_{0}+n}'(u) $.
	We will show that this step is well-defined.
	Suppose for some $ n < m $ there is  $ v, v' $ in $ \mathcal{D}_{n}, \mathcal{D}_{m} $ respectively such that $ w_{n-1}(u_{i}^{k_{0}+n})^{-1}v = w_{m-1}(u_{i}^{k_{0}+m})^{-1}v' $.
	Write $  w_{m-1} = w_{n-1}(u_{i}^{k_{0}+n})^{-1}u_{j}^{k_{0}+n} \cdots (u_{i}^{k_{0}+m-1})^{-1}u_{j}^{k_{0}+m-1} $. \newline
	Thus $ (u_{i}^{k_{0}+n})^{-1}v = (u_{i}^{k_{0}+n})^{-1}u_{j}^{k_{0}+n} \cdots (u_{i}^{k_{0}+m-1})^{-1}u_{j}^{k_{0}+m-1}(u_{i}^{k_{0}+m})^{-1}v'$. \newline
	Therefore $ v = u_{j}^{k_{0}+n} \cdots (u_{i}^{k_{0}+m-1})^{-1}u_{j}^{k_{0}+m-1}(u_{i}^{k_{0}+m})^{-1}v' $.
	As $ v $ does not contain $ u_{j}^{k_{0}+n} $ as a proper prefix, $ v =  u_{j}^{k_{0}+n} $, $ m = n + 1 $ and $ v' = u_{i}^{k_{0}+n+1} $.
	In this case, $ \mathcal{O}(w_{n-1}(u_{i}^{k_{0}+n})^{-1}u_{j}^{k_{0}+n}) =  \mathcal{O}_{k_{0}+n}'(v) = y $ and
	$  \mathcal{O}(w_{n}(u_{i}^{k_{0}+n+1})^{-1}u_{i}^{k_{0}+n+1}) = \mathcal{O}_{k_{0}+n+1}'(u_{i}^{k_{0}+n+1}) = y $. Thus $ \mathcal{O} $, so far as it has been defined is well-defined.

	To complete the construction of $ \mathcal{O} $, for $ u \in F $ with $ \mathcal{O}(u) $ not already defined, let $ u' $ be the largest prefix of $ u $ with $ \mathcal{O}(u') $ already
	defined. Such $ u' $ always exists as $ \mathcal{O}(e) $ is already defined.
	Then define $ \mathcal{O}(u) = \sigma_{u'^{-1}u}(\mathcal{O}(u')) $.
	Thus we have defined $ \mathcal{O}: F \rightarrow Y $.

	From the construction, it is easy to see that $ \mathcal{O} $ is a $ \delta_{k} $-pseudo orbit.
	Let $ x \in X $ $ \frac{1}{k} $ shadow $ \mathcal{O} $. By an analogous argument to \ref{CICTgeneral},
	$ d_{H}(\omega_{F_{w}}(x), Y) < \frac{1}{k} $. Thus $ Y \in IBT^{*}$.
\end{proof}

\begin{thm}\label{IBTgeneralmonoid}
	For an $H$-action on $X$ with $H$-shadowing, $IBT^{\circ}=\overline{\mathfrak{W}_{F_w}}$.
\end{thm}

\begin{proof}
	Let $ Y \in IBT^{\circ} $. For $ k \in \mathbb{N} $, find $ \delta_{k} $ such that any $ \delta_{k} $-pseudo orbit can $ \frac{1}{k} $ shadowed.	Define $ \mathcal{O} $ as follows.

	Fix $ k_{0} > \frac{1}{\delta_{k}} $. For $ n \geq k_{0} $ choose $ \left\{x_{1}^{n}, \dots, x_{k_{n}}^{n} \right\} \subseteq Y $ a $ \frac{1}{n} $ cover of $ Y $. By assumption,	we can find a $ \frac{1}{n} $-pseudo orbit $ \mathcal{O}_{n} $
	with indexes $ u_{y}^{n}, u_{m}^{n} $ such that $ \mathcal{O}_{n}(e) = \mathcal{O}_{n}(u^{n}_{y}) = y$, $ \mathcal{O}_{n}(u_{m}^{n}) = x_{m}^{n} $
	and $ u_{y}^{n} $ is not a prefix of 	$ u_{m}^{n} $ for all $ m $.

	Define $ w = u_{y}^{k_{0}}u_{y}^{k_{0}} \dots $. For notation, define $ u_{y}^{k_{0}-1} = e $.
	For $ u \in H $, find the maximal $ m $ such that $ u = u_{y}^{k_{0}-1} \dots u_{y}^{m}u' $ and define $ \mathcal{O}(u) = \mathcal{O}_{m+1}(u') $.
	Clearly, $ \mathcal{O} $ is a $ \delta_{k} $-H-pseudo orbit.

	By the same argument as above, choosing $ x $ to $ \frac{1}{k} $-shadow $ \mathcal{O} $ implies $ d_{H}(\omega_{F_{w}}(x), Y) < \frac{1}{k} $. Thus $ Y \in IBT^{\circ}$.
\end{proof}

\begin{definition}
	An \emph{asymptotic F-pseudo orbit} is a function $ \mathcal{O}: F \rightarrow X $ such that for every $ \delta > 0 $ there is an integer
	$ n $ such that for $ \left| w \right| > n $ and $ u\in S $, $ d(f_{u}(\mathcal{O}(w)), \mathcal{O}(wu)) < \delta $.
\end{definition}

\begin{definition}
	A function $ \mathcal{O}: F \rightarrow X $ is \emph{asymptotically shadowed} if there is an $ x \in X $ such that for every $ \epsilon > 0 $ there
	is an integer $ n $ such that for $ \left| w \right| > n  $ $ d(f_{w}(x), \mathcal{O}(w)) < \epsilon $.
\end{definition}

It is not the case that in a shift of finite type that every asymptotic pseudo orbit can be asymptotically shadowed.
Consider for instance the shift of finite type of $ \left\{0,1\right\}^{F} $ given by forbidding any 0 adjacent to a 1.
This shift of finite type has two elements: $ x_{0} $ and $ x_{1} $, the constant maps of 0 and 1 respectively.
We can construct an asymptotic pseudo orbit by first choosing $ j \in S $, then defining
$ \mathcal{O}(ju) = x_{0} $ for $ u \in F $, $ \mathcal{O}(iu) = x_{1} $ for $ i \neq j \in S $  and $ u \in F $, and $ \mathcal{O}(1) = x_{1} $.
Suppose that some $ y $ in the subshift asymptotically shadows $ \mathcal{O} $.
Then $ y $ must contain both 0 and 1, meaning it contains a 0 adjacent to a 1,
so $ y $ is not in the subshift. This contradicts $ \mathcal{O} $ being asymptotically shadowed.

\begin{thm}
	If $ X $ is an m-step SFT, every asymptotic, $ 2^{-m-1} $ pseudo-orbit can be asymptotically shadowed.
\end{thm}

\begin{proof}
	Let $ \mathcal{O} $ be such a pseudo orbit.
	Construct $ x $ by $ x(u) = \mathcal{O}(u)(1) $.
	By the previous theorem, $ x \in X $.

	For $ k > m + 1 $ find $ l_{k} $ such that for $ \left| u \right| > l_{k} $, $ d(\sigma_{i}(\mathcal{O}(u)), \mathcal{O}(ui)) < 2^{-k} $ for all $ i \in S $.
	We claim that for $ \left| u \right| > l_{k+1} + k $, $ d(\sigma_{u}(x), \mathcal{O}(u)) < 2^{-k} $.
	Notice by the choice of $ u $ that $ \mathcal{O}_{|u \Sigma^{k}} $ is a finite portion of a $ 2^{-k-1} $ pseudo orbit.
	Hence by a previous lemma, $ \mathcal{O}(u)(v) = \mathcal{O}(uv)(e) $ for $ v \in \Sigma^{k} $.
	Thus $ \sigma_{u}(x)(v) = x(uv) = \mathcal{O}(uv)(1) = \mathcal{O}(u)(v) $. Hence $ d(\sigma_{u}(x), \mathcal{O}(u)) < 2^{-k} $, so our claim
	is correct.
	Therefore $ x $ asymptotically shadows $ \mathcal{O} $.
\end{proof}

As it happens, this form of shadowing is sufficient for our purposes.

\begin{definition}
	A $G$-action on a compact metric spaces has the \emph{weak} $G$-asymptotic shadowing property if there exists $\delta>0$ such that every asymptotic $G$-pseudo-orbit which is also a $\delta$-$G$-pseudo-orbit  is asymptotically shadowed.
\end{definition}

\begin{thm}
	For an $G$-action on $X$ with weak $G$-asymptotic shadowing, $ \mathfrak{W}_{w}(X) = CICT $.
\end{thm}

\begin{proof}
	It remains to show that $ CICT \subseteq \mathfrak{W}_{w}(X) $.
	Let $ A \in CICT $ and find $ \delta > 0 $ that witnesses the asymptotic shadowing property.
	Choose $ k > \frac{1}{\delta} $ and let $ \mathcal{O} $ and $ w $ be as defined in Theorem \ref{CICTgeneral}.
	By construction, $ \mathcal{O} $ is an asymptotic $ \delta $-F-pseudo orbit. If $ x \in X $ asymptotically
	shadows $ \mathcal{O} $, it is not difficult to see that $ A = \omega_{w}(x) $.
\end{proof}

If instead we use the constructions given in Theorems \ref{IBTgeneralgroup} and \ref{IBTgeneralmonoid}, we obtain the following results.

\begin{thm}
	For an $G$-action on $X$ with weak $G$-asymptotic shadowing, \\ $ \mathfrak{W}_{F_{w}}(X) = IBT^{*} $.
\end{thm}

\begin{thm}
	For an $H$-action on $X$ with weak $H$-asymptotic shadowing, \\ $ \mathfrak{W}_{F_{w}}(X) = IBT^{\circ} $.
\end{thm}

%
%

\bibliographystyle{plain}
\bibliography{ComprehensiveBib}
\end{document}